\documentclass[11pt]{amsart}
\usepackage{amsmath}
\usepackage{amssymb}
\usepackage{graphicx}
\usepackage{mathrsfs}
\usepackage{tikz-cd}
\usepackage{xcolor}
\usepackage{tikz}
\usepackage[utf8]{inputenc}
\usepackage[english]{babel}
\usepackage{amsthm}

\newtheorem{thm}{Theorem}[section]
\newtheorem{prop}[thm]{Proposition}
\newtheorem{lemma}[thm]{Lemma}
\newtheorem{cor}[thm]{Corollary}

\theoremstyle{definition}
\newtheorem{defn}[thm]{Definition}
\newtheorem{rmk}[thm]{Remark}

\usetikzlibrary{%
  matrix,%
  calc,%
  arrows%
}

\graphicspath{ {Images/} }

\newcommand{\C}{\mathbb{C}}

\newcommand{\Z}{\mathbb{Z}}

\newcommand{\N}{\mathbb{N}}

\DeclareMathOperator{\Tr}{Tr}
\DeclareMathOperator{\id}{id}

\DeclareMathOperator{\im}{im}

\DeclareMathOperator{\Ext}{Ext}

\DeclareMathOperator{\Hom}{Hom}
\DeclareMathOperator{\Gl}{GL}

\DeclareMathOperator{\Sym}{Sym}
\DeclareMathOperator{\Alt}{Alt}
\DeclareMathOperator{\Sp}{Sp}
\usepackage{esint}
\usepackage[backref=page, colorlinks]{hyperref}
\title{A Cup Product Obstruction to Frobenius Stability}
\author{Forrest Glebe}

\begin{document}
\maketitle

\begin{abstract}
A countable discrete group $\Gamma$ is said to be Frobenius stable if a function from the group that is ``almost multiplicative'' in the point Frobenius norm topology is ``close'' to a genuine unitary representation in the same topology. The purpose of this paper is to show that if~$\Gamma$ is finitely generated and a non-torsion element of $H^2(\Gamma;\Z)$ can be written as a cup product of two elements in $H^1(\Gamma;\Z)$ then $\Gamma$ is not Frobenius stable. In general, 2-cohomology does not obstruct Frobenius stability. Some examples are discussed, including Thompson's group~$F$ and Houghton's group $H_3$. The argument is sufficiently general to show that the same condition implies non-stability in unnormalized Schatten $p$-norms for $1<p\le\infty$.
\end{abstract}
\section{Introduction}

Let $1\le p\le\infty$ and let $||\cdot||_p$ denote the unnormalized Schatten $p$-norm on the space of $k\times k$ complex matrices, $||M||_p=(\Tr((M^*M)^{p/2}))^{1/p}$ for $p<\infty$ and operator norm for $p=\infty$. Then a countable discrete group~$\Gamma$ is \textit{stable in the unnormalized Schatten $p$-norm} if for all sequences of functions~$\varphi_n$ from $\Gamma$ to the complex $k_n\times k_n$ unitary group $U_{k_n}$ the condition
\begin{equation}
||\varphi_n(gh)-\varphi_n(g)\varphi_n(h)||_p\rightarrow0,\,\,\forall g,h\in\Gamma\end{equation}
implies there exists a sequence of group homomorphisms $\psi_n:\Gamma\rightarrow U_{k_n}$ so that
\begin{equation}   
||\psi_n(g)-\varphi_n(g)||_p\rightarrow0,\,\,\forall g\in\Gamma.\end{equation}
Of particular interest is the $p=2$ case, called {\em Frobenius stability}, and the $p=\infty$ case called {\em matricial stability}. We will call a sequence of functions~$(\varphi_n)$ that satisfies condition (1) an {\em asymptotic homomorphism}. If there exist homomorphisms~$(\psi_n)$ satisfying condition~(2) we say that $(\varphi_n)$ is {\em perturbable to homomorphisms}. Frobenius stability was introduced by de Chiffre, Glebsky, Lubotzky, and Thom in~\cite{nonaproximable}. Stability of a finitely presented group is equivalent to a notion of stability of the presentation of that group; this notion was shown to be independent of the presentation by Arzhantseva and Păunescu in~\cite{almostperm}.

The goal of this paper is to show the following. 

\begin{thm}\label{main}
Let $\Gamma$ be a finitely generated discrete group, and let $1<p\le\infty$. If there are $\alpha,\beta\in H^1(\Gamma;\Z)$ so that $\alpha\smile\beta\in H^2(\Gamma;\Z)$ is non-torsion then $\Gamma$ is not stable in the unnormalized Schatten $p$-norm.
\end{thm}

In general nonzero second Betti number does not obstruct Frobenius stability; in~\cite{lattice} Bader, Lubotzky, Sauer, and Weinberger show that $\Sp_{2n+2}(\Z)$ is Frobenius stable despite having a nonzero second Betti number.

In the $p=\infty$ case our result follows from the methods developed by Dadarlat in~\cite{obs}, though it does not follow from the main result as stated there. The techniques we use here are more similar to those developed by the author in~\cite{constructive} and~\cite{frob}. In~\cite{frob} the notion of a \textit{skinny} cohomology class is used as an obstruction to Frobenius stability of nilpotent groups. Here a cohomology class $[\sigma]$ is skinny with respect to a homomorphism $\alpha:\Gamma\rightarrow\Z$ if the restriction of $[\sigma]$ to the kernel of $\alpha$ is a coboundary. Then $\alpha\smile\beta$ is skinny with respect to both $\alpha$ and $\beta$, motivating our main result.

The basic idea of the proof is that since $\alpha,\beta\in H^1(\Gamma;\Z)\cong\Hom(\Gamma,\Z)$ we can view the pair $(\alpha,\beta)$ as a homomorphism from $\Gamma$ to $\Z^2$. We use a classic example of an asymptotic homomorphism of $\Z^2$ due to Voiculescu in~\cite{voiculescuM} and pull it back by $(\alpha,\beta)$ to get an asymptotic homomorphism of $\Gamma$. We show that this is a projective representation\footnote{Meaning a map from $\Gamma$ to unitaries whose failure to be multiplicative, $\rho(gh)\rho(h)^{-1}\rho(g)^{-1}$, is scalar-valued.} of $\Gamma$ (Lemma~\ref{twist}). Then we use a winding number type argument based on the nontriviality of $\alpha\smile\beta$ to show that this asymptotic representation is not perturbable (Lemma~\ref{general}).

In Section 4 we go over examples of groups that Theorem~\ref{main} applies to. We show that Thompson's group $F$ and Houghton's groups $H_n$ for $n\ge3$ among other examples are not stable in the unnormalized Schatten $p$-norm for $p>1$. In some cases, there is a more direct elementary argument that the group is not stable because the map $(\alpha,\beta)$ from the group to $\Z^2$ splits; this is the case for Thompson's group $F$ in particular. In many cases the splittings are not obvious, so the main result is still useful in identifying the groups as non-stable. 

An asymptotic representation that is not perturbable can be described as follows.
\begin{defn}\label{rhon}
Since $H^1(\Gamma;\Z)\cong\Hom(\Gamma,\Z)$ we may view $\alpha$ and $\beta$ as group homomorphisms from $\Gamma$ to $\Z$. Then define

$$\rho_n(g)=u_n^{\alpha(g)}v_n^{\beta(g)}$$
where $u_n$ and $v_n$ are the $n\times n$ Voiculescu unitaries
$$u_n=\begin{bmatrix}
0&0&\cdots&0&0&1\\
1&0&\cdots&0&0&0\\
0&1&\cdots&0&0&0\\
\vdots&\vdots&\ddots&\vdots&\vdots&\vdots\\
0&0&\cdots&1&0&0\\
0&0&\cdots&0&1&0
\end{bmatrix}\mbox{ and }
v_n=\begin{bmatrix}
\exp\left(\frac{2\pi i}{n}\right)&0&0&\cdots&0\\
0&\exp\left(\frac{4\pi i}{n}\right)&0&\cdots&0\\
0&0&\exp\left(\frac{6\pi i}{n}\right)&\cdots&0\\
\vdots&\vdots&\vdots&\ddots&\vdots\\
0&0&0&\cdots&1
\end{bmatrix}.
$$
 
\end{defn}

\section{Notation}

There are many ways to characterize group homology and cohomology; we will give a concrete description of 1-cohomology, 2-cohomology, and 2-homology here. We will only use homology and cohomology with coefficients in $\Z$ and the trivial action in this paper. For more about this construction see \cite[Chapter II.3]{coho}.

As stated in the introduction, $H^1(\Gamma;\Z)\cong\Hom(\Gamma,\Z)$ and we can take this to be the definition.  

\begin{defn}
We define a {\em 2-cocycle} to be a function $\sigma$ from $\Gamma^2$ to $\Z$ satisfying the following equation
$$\sigma(g,h)-\sigma(g,hk)+\sigma(gh,k)-\sigma(h,k)=0.$$
A {\em 2-coboundary} is a function that can be written in the form
$$\sigma(g,h)=\gamma(g)-\gamma(gh)+\gamma(h)$$
for some function $\gamma:\Gamma\rightarrow\Z$. Every 2-coboundary is a 2-cocycle and $H^2(\Gamma;\Z)$ is defined to be the group of 2-cocycles, mod the subgroup of 2-coboundaries. The group operation is pointwise addition.
\end{defn}

\begin{defn}
Define $C_k(\Gamma)$ to be formal linear combinations of elements of~$\Gamma^k$. We write a typical element of $C_2(\Gamma)$ as
$$\sum_{j=1}^Nx_j[a_j|b_j]$$
with $a_j,b_j\in\Gamma$ and $x_j\in\Z$. Define $\partial_2:C_2(\Gamma)\rightarrow C_1(\Gamma)$ to by the equation
$$\partial_2[a|b]=[a]-[ab]+[b]$$
and $\partial_3:C_3(\Gamma)\rightarrow C_2(\Gamma)$ by
$$\partial_3[a|b|c]=[a|b]-[a|bc]+[ab|c]-[b|c].$$
Then $H_2(\Gamma;\Z):=\ker(\partial_2)/\im(\partial_3)$. An element of $\ker(\partial_2)$ is referred to as a {\em 2-cycle} and an element in $\im(\partial_3)$ is referred to as a {\em 2-boundary}.
\end{defn}

\begin{defn}
The {\em Kronecker pairing} between 2-homology and 2-cohomology is a bilinear map from $H^2(\Gamma;\Z)\times H_2(\Gamma;\Z)$ to $\Z$ defined by the formula
$$\left\langle\sigma,\sum_{j=1}^Nx_n[a_j|b_j]\right\rangle:=\sum_{j=1}^Nx_j\sigma(a_j,b_j)$$
where $\sigma$ is a cocycle, and $\sum_{j=1}^Nx_n[a_j|b_j]$ is a cycle. The value does not depend on either choice of representative.
\begin{defn}
The {\em cup product} is a bilinear map $\cdot\smile\cdot$ from $H^j(\Gamma;\Z)\times H^{k}(\Gamma;\Z)$ to $H^{j+k}(\Gamma;\Z)$. We will write the definition for the case that $j=k=1$. If $\alpha,\beta\in H^1(\Gamma;\Z)\cong\Hom(\Gamma,\Z)$ we define $\alpha\smile\beta$ to the cohomology class of the cocycle
$$\sigma(g,h)=\alpha(g)\beta(h).$$
For a more general definition, and more information see~\cite[Chapter V.3]{coho}.
\end{defn}

\begin{prop}
Let $q:\Gamma\rightarrow\Lambda$ be a group homomorphism. Then there is a map $q^*:H^*(\Lambda;\Z)\rightarrow H^*(\Gamma;\Z)$. Moreover $q^*(\alpha\smile\beta)=q^*(\alpha)\smile q^*(\beta)$.
\end{prop}

\begin{defn}[\cite{constructive} Definition 3.3]\label{pairing}
If
$$c=\sum_{j=1}^N x_j[a_j|b_j]\in C_2(\Gamma)$$
and $\rho:\Gamma\rightarrow\Gl_n(\C)$ so that for all $j\in\{1,\ldots,N\}$
$$||\rho(a_jb_j)\rho(a_j)^{-1}\rho(b_j)^{-1}-\id_{\C^n}||_\infty<1$$
we define
$$\langle \rho,c\rangle=\frac{1}{2\pi i}\sum_{j=1}^Nx_j\Tr(\log(\rho(a_jb_j)\rho(b_j)^{-1}\rho(a_j)^{-1}))$$
where $\log$ is defined as a power series centered at 1.
\end{defn}

If $\partial c=0$ we have that $\langle\rho,c\rangle\in\Z$  \cite[Proposition 3.4]{constructive}. The version of the ``winding number argument'' we are using is as follows.

\begin{thm}\label{obstruction}\cite[Theorem 3.7]{constructive}
If $\rho_0$ is a (not necessarily unitary) representation of $\Gamma$, $\rho_1$ is a function from~$\Gamma$ to~$U(n)$, 
$$c=\sum_{j=1}^N x_j[a_j|b_j]$$
is a 2-cycle on $\Gamma$, and $$||\rho_1(g)-\rho_0(g)||_\infty<\frac{1}{24}$$
for all $g\in\{a_j,b_j,a_jb_j\}_{j=1}^N$ then $\rho_1$ is multiplicative enough for $\langle\rho_1,c\rangle$ to be defined and $\langle\rho_1,c\rangle=0$.
\end{thm}
This argument has its roots in the ``winding number argument'' discovered by Kazhdan~\cite{KvoiculescuM}, and later independently by Exel and Loring~\cite{ELVoiculescuM}.

\end{defn}

\section{Proofs}

\begin{lemma}\label{twist}
Define $\rho_n$ as in Definition~\ref{rhon}. It obeys the following identity
$$\rho_n(gh)\rho_n(h)^{-1}\rho_n(g)^{-1}=\exp\left(-\frac{2\pi i}{n}\beta(g)\alpha(h)\right)\id_{\C^n}.$$
\begin{proof}
Let $\omega=\exp\left(\frac{2\pi i}{n}\right)$ and note the four identities $v_nu_n=\omega u_nv_n$, $v_n^{-1}u_n=\omega^{-1}u_nv_n^{-1}$, $v_nu_n^{-1}=\omega^{-1}u_n^{-1}v_n$, and $v_n^{-1}u_n^{-1}=\omega u_n^{-1}v_n^{-1}$. From these four identities, it follows that $v_n^xu_n^y=\omega^{xy}u_n^xv_n^y$, for all $x,y\in\Z$. We compute
\begin{align*}
\rho_n(gh)\rho_n(h)^{-1}\rho_n(g)^{-1}&=u_n^{\alpha(g)+\alpha(h)}v_n^{\beta(g)+\beta(h)}v_n^{-\beta(h)}u_n^{-\alpha(h)}v_n^{-\beta(g)}u_n^{-\alpha(g)}\\
&=u_n^{\alpha(g)+\alpha(h)}v_n^{\beta(g)}u_n^{-\alpha(h)}v_n^{-\beta(g)}u_n^{-\alpha(g)}\\
&=\omega^{-\beta(g)\alpha(h)}u_n^{\alpha(g)+\alpha(h)}u_n^{-\alpha(h)}v_n^{\beta(g)}v_n^{-\beta(g)}u_n^{-\alpha(g)}\\
&=\omega^{-\beta(g)\alpha(h)}\id_{\C^n}.
\end{align*}
\end{proof}
\end{lemma}
\begin{lemma}\label{general}
Let $\Gamma$ be a countable discrete group and let $[\sigma]\in H^2(\Gamma;\Z)$ be a cohomology class with cocycle representative $\sigma$. Suppose that $\rho_n$ is a sequence of functions from $\Gamma$ to $k_n\times k_n$ unitaries so that
$$\rho_n(gh)\rho_n(h)^{-1}\rho_n(g)^{-1}=\exp\left(\frac{2\pi i}{n}\sigma(g,h)\right)\id_{\C^n}.$$
Then
$$||\rho_n(g)\rho_n(h)-\rho_n(gh)||_\infty\le2\pi|\sigma(g,h)|/n;$$
in particular, $\rho_n$ is asymptotically multiplicative in operator norm. 

Now suppose that $[\sigma]$ pairs nontrivially with some homology class. Then~$\rho_n$ is not perturbable to homomorphisms in operator norm.
\end{lemma}

\begin{proof}
This follows the proof of \cite[Theorem~3.20]{constructive}.

The first part follows from the fact that
$$||\rho_n(gh)-\rho_n(g)\rho_n(h)||_\infty=||\rho_n(gh)\rho_n(h)^{-1}\rho_n(g)^{-1}-\id_{\C^{k_n}}||_\infty.$$

Now we will show that for large enough $n$, $\rho_n$ is not close to any genuine representation of $\Gamma$ in operator norm on a particular finite subset of $\Gamma$. There is some 2-cycle $c\in C_2(\Gamma)$ written
$$c=\sum_{i=1}^Nx_i[a_i|b_i].$$
so that
$$\langle \sigma,c\rangle\ne0.$$
Then we compute that
\begin{align*}
\langle\rho_n,c\rangle&=\frac{1}{2\pi i}\sum_{j=1}^N x_j\Tr(\log(\rho_n(a_jb_j)\rho_n(b_j)^{-1}\rho_n(a_j)^{-1}))\\
&=\frac{1}{2\pi i}\sum_{j=1}^N x_j\Tr\left(\log\left(\exp\left(\frac{2\pi i}{n}\sigma(a_j,b_j)\right)\id_{\C^{k_n}}\right)\right)\\
&=\frac{1}{2\pi i}\sum_{j=1}^N x_j\frac{2\pi i}{n}\sigma(a_j,b_j)\Tr(\id_{\C^{k_n}})\\
&=\langle\sigma,c\rangle\frac{k_n}{n}\\
&\ne0.
\end{align*}
By Theorem~\ref{obstruction} it follows that $\rho_n$ cannot be within $\frac{1}{24}$ of a genuine representation on the set $\{a_j,b_j,a_jb_j\}_{j=1}^N$ and thus cannot be perturbed to a genuine representation.

\end{proof}

Now we are ready to prove Theorem~\ref{main}.
\begin{proof}
Note that by the definition of the cup product $\sigma(g,h)=-\beta(g)\alpha(h)$ is a cocycle representative of the cohomology class $-\beta\smile\alpha=\alpha\smile\beta$~\cite[V.3 equation 3.6]{coho}. 

First, we show asymptotic multiplicativity. In operator norm, this follows directly from Lemma~\ref{twist} and Lemma~\ref{general}. For the $p<\infty$ case, we use the same lemmas to show
\begin{align*}
||\rho_n(gh)-\rho_n(g)\rho_n(h)||_p&\le||\rho_n(gh)-\rho_n(g)\rho_n(h)||_\infty\cdot n^{1/p}\\
&\le2\pi|\alpha(g)\beta(h)|\cdot n^{1/p-1}.
\end{align*}
This goes to zero for $p>1$.

Since $\alpha\smile\beta$ is non-torsion it must also pair nontrivially with a 2-homology class; to see this note that from the universal coefficient theorem \cite[Theorem~53.1]{eat}  we have a short exact sequence
$$\begin{tikzcd}
0\arrow{r}{}&\Ext(H_1(\Gamma;\Z),\Z)\arrow{r}{}& H^2(\Gamma;\Z)\arrow{r}{}&\Hom(H_2(\Gamma;\Z),\Z)\arrow{r}{}&0
\end{tikzcd}.$$
Since $\Gamma$ is finitely generated $H_1(\Gamma;\Z)\cong\Gamma/[\Gamma,\Gamma]$ \cite[page 36]{coho} is finitely generated as well. Thus $\Ext(H_1(\Gamma;\Z),\Z))$ can be show to be torsion from \cite[Theorem 52.3]{eat} and the table on \cite{eat} page~331. 

Now it follows from Lemma~\ref{twist} and Lemma~\ref{general} that $\rho_n$ is not perturbable in operator norm to a sequence of genuine representations. Since the operator norm is smaller than any other unnormalized Schatten $p$-norm it follows that $\rho_n$ cannot be close to a genuine representation in any of these norms either.
\end{proof}
\begin{rmk}
Note that the finitely generated condition in Theorem~\ref{main} can be dropped if we require that $\alpha\smile\beta$ pairs nontrivially with a 2-homology class.
\end{rmk}

\section{Examples}

Note that our result relies on the existence of a homomorphism $(\alpha,\beta)$ from~$\Gamma$ to $\Z^2$. Because $\alpha\smile\beta=-\beta\smile\alpha$~\cite[V.3 equation 3.6]{coho} our assumption that $\alpha\smile\beta$ is non-torsion implies that $\alpha$ and $\beta$ must be linearly independent. If the map $(\alpha,\beta)$ is a surjection that splits, our theorem applies, but in this case, the non-stability can be proven with a simpler argument. Since the argument is more general we will explain a more general context.

\begin{defn}[\cite{nonaproximable}]
If $\{(G_k,d_k)\}$ is a family of groups with bi-invariant metrics, $(G_k,d_k)${\em-stability} is defined analogously go stability in the unnormalized Schatten $p$-norm but $U_{k_n}$ is replaced by the family $G_{k_n}$ and the convergence in both conditions is replaced by convergence in the metric $d_{k_n}$. {\em Uniform $(G_k,d_k)$-stability} is defined analogously, but with uniform convergence instead of pointwise convergence.
\end{defn}

If the map $(\alpha,\beta):\Gamma\rightarrow\Z^2$ splits then the non-stability of $\Gamma$ can be proved with the following lemma pointed out to the author by Francesco Fournier-Facio.

\begin{lemma}\label{split}
Suppose that $\Lambda\rtimes\Upsilon$ is (uniformly) $(G_k,d_k)$-stable. Then $\Upsilon$ is (uniformly) $(G_k,d_k)$-stable as well.
\end{lemma}

\begin{proof}
Call $\Gamma=\Lambda\rtimes\Upsilon$ and let $r:\Gamma\rightarrow\Upsilon$ and $s:\Upsilon\rightarrow\Gamma$ be the obvious maps. Let $\rho_n$ be a sequence of (uniformly) asymptotically multiplicative maps from~$\Upsilon$ to~$G_{k_n}$. Then $\rho_n\circ r$ are (uniformly) asymptotically multiplicative as well. Thus there is a sequence of representations $\pi_n:\Gamma\rightarrow G_{k_n}$ be a sequence of representations that (uniformly) approximate $\rho_n\circ r$. Then $\pi_n\circ s$ (uniformly) approximates $\rho_n\circ r\circ s=\rho_n$.
\end{proof} 

Many examples of groups that satisfy the conditions of Theorem~\ref{main} have a split homomorphism to $\Z^2$, but not all; the easiest counterexample is higher genus surface groups. These satisfy the conditions of the theorem, but they are hyperbolic so $\Z^2$ cannot be embedded as a subgroup in them by~\cite[Corollary III.$\Gamma$~3.10]{metriccurv}. They are already known not to be stable in unnormalized Schatten $p$-norm for $1<p\le\infty$; it follows from the computation in the proof of Theorem~2 in~\cite{KvoiculescuM}. In our examples, we will prove that groups are not stable using Theorem~\ref{main}, then use Lemma~\ref{split} to provide an alternate proof if possible. Often the splittings are not obvious, so Theorem~\ref{main} is still useful in identifying these groups as non-stable; often the proof that the group satisfies the cohomological condition is shorter than the proof that the map splits.

Many interesting examples come from extensions of a group that fits the conditions of Theorem~\ref{main} by a locally finite group, a class of groups that was suggested to the author by Francesco Fournier-Facio.

\begin{lemma}
Suppose $\Lambda\trianglelefteq\Gamma$ and let $q:\Gamma\rightarrow\Gamma/\Lambda$ be the quotient map. Suppose that $H^1(\Lambda;\Z)=\{0\}$. Let $\alpha,\beta\in H^2(\Gamma/\Lambda;\Z)$. If $\alpha\smile\beta$ is non-torsion it follows that $q^*(\alpha)\smile q^*(\beta)$ is non-torsion as well.
\end{lemma}

\begin{proof}
From the exact sequence in~\cite{gysin} Theorem~III.2 we see that the map $q^*:H^2(\Gamma/\Lambda;\Z)\rightarrow H^2(\Gamma;\Z)$ is an injection. Because $q^*$ preserves cup product we must have that $q^*(\alpha)\smile q^*(\beta)=q^*(\alpha\smile\beta)$ is non-torsion.
\end{proof}

\begin{rmk}\label{extension}
Note that since $H^1(\Lambda;\Z)\cong\Hom(\Lambda;\Z)$ it follows that if all elements of $\Lambda$ are torsion then $H^1(\Lambda;\Z)=\{0\}$. Thus if we have an extension
$$\begin{tikzcd}
e\arrow{r}{}&\Lambda\arrow{r}{}& \Gamma\arrow{r}{}&\Upsilon\arrow{r}{}&e
\end{tikzcd}$$
where $\Lambda$ has only torsion elements, $\Gamma$ is finitely generated, and $\Upsilon$ meets the conditions of Theorem~\ref{main} then $\Gamma$ meets the condition of Theorem~\ref{main} as well. The same argument holds if $\Lambda$ has property (T), is simple, or is any other group without 1-cohomology.
\end{rmk}

\begin{defn}[\cite{houghton}]
The Houghton groups $H_n$ for $n\ge3$ are defined as follows. Let $X_n=\{1,\ldots,n\}\times\N$. For $k\in\{2,\ldots,n\}$ define 
$$g_k(x,y)=\begin{cases}(x,y+1)&x=1;\\
(1,0)&(x,y)=(k,0);\\
(x,y-1)&x=k\mbox{ and }y\ne0;\\
(x,y)&\mbox{otherwise.}
\end{cases}$$
Then $H_n$ is the subgroup of permutations of $X_n$ spanned by $g_2,\ldots,g_n$.
\end{defn}

\begin{cor}
For $n\ge3$, and $1<p\le\infty$ the Houghton group $H_n$ is not stable in the unnormalized Schatten $p$-norm.
\end{cor}

\begin{proof}
It follows from an argument originally due to \cite{houghtonext} (see~\cite{houghtonext2} for a full explanation) that there is an extension
$$\begin{tikzcd}
e\arrow{r}{}&S_\infty\arrow{r}{}& H_n\arrow{r}{}&\Z^{n-1}\arrow{r}{}&e
\end{tikzcd}$$
so by Remark~\ref{extension} and Theorem~\ref{main} it follows that $H_n$ is not stable in the Schatten $p$-norm.
\end{proof}

For $n\ge4$ one can find a way to show this with Lemma~\ref{split} instead of Theorem~\ref{main}. If $g_2,\ldots, g_n$ are the generators above. Note that for $i$, $j$, $k$ distinct the elements $g_i^{-1}g_j$ and $g_k$ commute with each other. So the map from $H_n\rightarrow\Z^2$ defined by $g_j\mapsto e_1$, $g_k\mapsto e_2$, and $g_\ell\mapsto0$ for $\ell\ne j,k$ has a splitting determined by $e_1\mapsto g_i^{-1}g_j$ and $e_2\mapsto g_k$. The author does not see an obvious splitting in the $n=3$ case.

\begin{cor}\label{thompson}
For $1<p\le\infty$ Thompson's group $F$ is not stable in the unnormalized Schatten $p$-norm.
\end{cor}
\begin{proof}
The cohomology ring of $F$ can be computed with methods in~\cite{thompsomHomologyLong}; see~\cite{thompsonHomology} for an explicit computation in the case of $F$. From this computation, one can see that $H^1(F;\Z)\cong\Z^2$ and the cup product of the two generators is non-torsion.
\end{proof}

A similar question of whether or not $F$ is permutation stable was raised by Arzhantseva and Păunescu in~\cite{almostperm}, where they point out that if $F$ were permutation stable it would imply that it is not sofic, which would demonstrate both the existence of a non-sofic group and that $F$ is non-amenable, both of which are famous unsolved problems. Showing that $F$ is stable in the {\em normalized} Hilbert-Schmidt norm would also be very interesting because it would show that $F$ is not hyperlinear by the same argument they give. Our result is an interesting contrast to the result of Fournier-Facio and Rangarajan in~\cite{ulamlamp} that $F$ is uniformly stable with respect to all submultiplicative norms.

An alternate proof of Corollary~\ref{thompson} fact (using Lemma~\ref{split}) was pointed out to the author by Francesco Fournier-Facio. View $F$ as the group piecewise-linear homeomorphisms from the unit interval to itself where all slopes are powers of 2 and all non-differentiable points are dyadic rationals. The group operation is composition. Let $f,g\in F$ be any two functions so that $f|_{[0,\frac12]}=\id_{[0,\frac12]}$ and $g|_{[\frac12,1]}=\id_{[\frac12,1]}$, but $f'(1)=g'(0)=2$. One can define homomorphisms $\alpha,\beta:F\rightarrow\Z$ by $\alpha(h)=\log_2(h'(1))$ and $\beta(h)=\log_2(h'(0))$. Then sending the generators of $\Z^2$ to $f$ and $g$ provides a splitting for the map $(\alpha,\beta)$.

Another class of examples are groups of the form $\Gamma\vee\Lambda$. Here $\Gamma\vee\Lambda$ is defined as follows. 
\begin{defn}[\cite{vee}]\label{vee}
Pick set models of $\Gamma$ and $\Lambda$ so that the identity elements of $\Gamma$ and $\Lambda$ are identified with each other, but $(\Gamma\setminus\{e\})\cap(\Lambda\setminus\{e\})=\emptyset$. Then let $S$ be the set $\Gamma\cup\Lambda$. Then define an action of $\Gamma$ on $S$ so that $\Gamma$ acts on itself by left translation while it stabilizes elements of $\Lambda\setminus\{e\}$ and defines an action of~$\Lambda$ on $S$ analogously. Define $\Gamma\vee\Lambda$ to be the subgroup of $\Sym(S)$ generated by the actions of $\Gamma$ and $\Lambda$.
\end{defn}

One can show that if $H^1(\Gamma;\Z)$ and $H^1(\Lambda;\Z)$ are both nontrivial, then $\Gamma\times\Lambda$ is not stable in the unnormalized Schatten $p$-norm, for $1<p\le\infty$, by applying Lemma~\ref{split} as follows. By the cohomological assumption, there are maps $\alpha:\Gamma\rightarrow\Z$ and $\beta:\Lambda\rightarrow\Z$, which we may take to be surjections. Then the map $(\alpha,\beta):\Gamma\times\Lambda\rightarrow\Z^2$ clearly splits. From this it follows that $\alpha\smile\beta$ is non-torsion.

\begin{cor}
Let $1<p\le\infty$. If $\Gamma$ and $\Lambda$ are finitely generated groups so that the integer cohomology of each is nontrivial then $\Gamma\vee\Lambda$ is not stable in the unnormalized Schatten $p$-norm.
\end{cor}

\begin{proof}
Note that both groups must be infinite, in order to have integer cohomology. Then by \cite[Proposition~1.7]{vee}, for any infinite groups $\Gamma$ and $\Lambda$ there is a short exact sequence
$$\begin{tikzcd}
e\arrow{r}{}&\Alt_f(S)\arrow{r}{}&\Gamma\vee\Lambda\arrow{r}{\pi}&\Gamma\times\Lambda\arrow{r}{}&e
\end{tikzcd}$$
where $S$ is the set defined in Definition~\ref{vee} and $\Alt_f(S)$ is the group of finitely-supported even permutations of $S$. From Remark~\ref{extension} the result follows.
\end{proof}

This can also be explained with Lemma~\ref{split}. While by \cite[Theorem~1.8]{vee} the map from $\Gamma\times\Lambda\rightarrow\Gamma\vee\Lambda$ often does not split, the composition with the map to $\Z^2$ does. If $\alpha$ and $\beta$ are surjective elements in $\Hom(\Gamma,\Z)$ and $\Hom(\Lambda,\Z)$ respectively then let $s:\Z^2\rightarrow\Gamma\times\Lambda$ be a splitting. By \cite[Proposition~7.1]{vee} this also induces an, $s_\vee$ from $\Z\vee\Z$ into $\Gamma\vee\Lambda$. By \cite[Proposition~4.2]{vee} the map, $\pi_{\Z}:\Z\vee\Z\rightarrow\Z^2$ does split; call the splitting $t$. Then we claim that $s_\vee\circ t$ is a splitting of $(\alpha,\beta)\circ\pi$. To show this compute
\begin{align*}
(\alpha,\beta)\circ\pi\circ s_\vee\circ t&=(\alpha,\beta)\circ s\circ\pi_\Z\circ t\\
&=\id_{\Z^2}.
\end{align*}

\begin{defn}[\cite{wreathlike1}]
Suppose $\Upsilon$ and $\Lambda$ are groups, and $I$ is a set that~$\Upsilon$ acts on. Then a {\em wreath-like product} of $\Lambda$ by $\Upsilon$ with corresponding to the action is a group $\Gamma$ fitting into an extension
$$\begin{tikzcd}
e\arrow{r}{}&\bigoplus_{i\in I}\Lambda_i\arrow{r}{}&\Gamma\arrow{r}{\pi}&\Upsilon\arrow{r}{}&e
\end{tikzcd}$$
where $\Lambda_i\cong\Lambda$, and for all $g\in\Gamma$, $g\Lambda_ig^{-1}=\Lambda_{\pi(g)\cdot i}$.
\end{defn}

\begin{cor}
If $\Lambda$ is a finitely generated group with trivial 1-cohomology, and $\Upsilon$ satisfies the conditions of Theorem~\ref{main} then so does any wreath-like product of $\Lambda$ by~$\Upsilon$ with a transitive action. In particular, a wreath-like product of a finite group, property (T) group, or finitely generated simple group by~$\Z^n$ for $n\ge2$, with a transitive action, is not stable in the Schatten $p$-norm for $1<p\le\infty$.
\end{cor}
\begin{proof}
Note that any nonzero homomorphism from $\bigoplus_{i\in I}\Lambda_i$ must pull back to a nonzero homomorphism from $\Lambda_i$ to $\Z$ for some $i$. It follows that $H^1\left(\bigoplus_{i\in I}\Lambda_i;\Z\right)=\{0\}$. Now the result follows from the definition and Remark~\ref{extension}, if we can show that the wreath-like product is finitely generated. To this end let $\Lambda$, $\Gamma$, and $\Upsilon$ be as above. Let $g_1,\ldots, g_n\in \Gamma$ so that $\pi(g_1),\ldots,\pi(g_n)$ generate $\Upsilon$, pick a fixed $j\in I$, and let $h_1,\ldots, h_m$ generate~$\Lambda_j$. We claim that $g_1,\ldots,g_k,h_1,\ldots,h_m$ generates $\Gamma$. Let $G\le\Gamma$ be the subgroup generated by $g_1,\ldots, g_k$, and note that by definition $\pi(G)=\Upsilon$. Now note that for all $i\in I$ there is some $y\in\Upsilon$ so that $y\cdot j=i$. Then letting $g\in G$ so that $\pi(g)=y$ we have that $g\Lambda_jg^{-1}=\Lambda_{\pi(g)\cdot j}=\Lambda_i$. Thus $\bigoplus_{i\in I}\Lambda_i$ is in the subgroup generated by $g_1,\ldots,g_k\cup\Lambda_j$. Now let $h\in\Gamma$. Note that there is $h'\in G$ so that $\pi(h)=\pi(h')$. Then $h^{-1}h'\in\bigoplus_{i\in I}\Lambda_i$ completing the proof.
\end{proof}

\begin{center}
\textbf{Acknowledgments}
\end{center}
I would like to thank Francesco Fournier-Facio for many comments that shaped the Examples section. Most of the changes from the first version of this paper to the second came from a conversation I had with him. He explained a more elementary proof that Thompson's group $F$ is not stable and suggested the other examples that appear in the paper. I would like to thank Eli Bashwinger for answering a question I had about the cohomology of Thompson and Thompson-like groups; his answer inspired me to make this note. I would also like to thank the Purdue math department for supporting me with the Ross-Lynn grant during the 2023-2024 academic year.
\bibliographystyle{plain}
\bibliography{main}
\end{document}